\documentclass[12pt,leqno]{amsart}
\usepackage{amscd}
\usepackage{amssymb}
\usepackage{amsfonts}
\usepackage{latexsym}
\usepackage{amsthm}
\newcommand{\norm}[1]{\left\lVert #1\right\rVert}
\newcommand{\abs}[1]{\left\lvert #1\right\rvert}                                                                                                                                                                           
\newcommand{\bigabs}[1]{\bigl\lvert #1\bigr\rvert}

\newcommand{\Nil}{\rm{Nil}}
\newcommand{\Sol}{\rm{Sol}}

\numberwithin{equation}{section}
\theoremstyle{plain}
\newtheorem{theorem}{Theorem}[section]

\newtheorem{lemma}[theorem]{Lemma}

\newtheorem{cor}[theorem]{Corollary}

\newcommand\R{\mathbb{R}}

\newcommand{\bT}{\mathbb{T}}
\newcommand{\eE}{\mathrm{e}}

\newcommand{\g}{\mathfrak{g}}

\begin{document}
\parskip2pt

\title[The non-Lagrangian case]{The Calabi--Yau equation for $T^2$-bundles over $\mathbb{T}^2$: the non-Lagrangian case}

\author{ \ Ernesto Buzano, \ Anna Fino,\ Luigi Vezzoni}

\begin{abstract}
In the spirit of \cite{TW,FLSV}, we study the Calabi-Yau equation on    $T^2$-bundles over  $\mathbb{T}^2$ endowed with an invariant non-Lagrangian almost-K\"ahler structure showing that for $T^2$-invariant initial data it  reduces to a Monge-Amp\`ere equation having a unique solution. In this way we prove that for  every total space $M^4$ of an orientable   $T^2$-bundle over  $\mathbb{T}^2$ endowed with an   invariant  almost-K\"ahler structure    the Calabi-Yau problem has a solution for every normalized $T^2$-invariant  volume form.
\end{abstract}

\maketitle

\section{Introduction}

Let  $(M^{2n}, J, \Omega)$  be a $2n$-dimensional compact K\"ahler manifold with associated complex structure  $J$ and symplectic form  $\Omega$.
In view of a  celebrated Yau's  theorem \cite{Yau} for every  volume form $\sigma$  on $M^{2n}$ satisfying
\begin{equation} \label{normvol} 
\int_{M^{2n}}\Omega^n=\int_{M^{2n}} \sigma
\end{equation}
 there exists a unique K\"ahler form $\tilde{\Omega}$ in the same de Rham cohomology class of
$\Omega$  and such that
\begin{equation}\label{CY}
\tilde{\Omega}^n =\sigma\,.
\end{equation}
Equation \eqref{CY} still makes sense in the {\em almost-K\"ahler} context when $J$ is merely an almost-complex structure and $\Omega$ remains closed. The  almost complex structure $J$  is still
orthogonal relative to a Riemannian metric $g$ for which
$\Omega(X,Y)=g(JX,Y)$, and
\begin{equation} \label{dalpha}
\tilde \Omega =\Omega +d\alpha
\end{equation}
is again assumed to be a positive-definite $(1,1)$-form relative to
$J$. 
In this context  the equations \eqref{normvol}, \eqref{CY}  and \eqref{dalpha} constituite   the {\em Calabi-Yau  problem}, which in the last years  has been  intensively studied in four dimensions (see \cite{Donaldson,TWY,TW,FLSV} and the references therein).

 In \cite{Donaldson} Donaldson introduced the Calabi-Yau problem  for almost-K\" ahler manifolds showing that
equation \eqref{CY} has unique solution in dimensions four  and that it is related to some  other central  problems in symplectic geometry. In \cite{TWY} Tosatti, Weinkove and Yau gave a sufficient condition  for the existence of solution to the Calabi-Yau equation in terms of the Chern connection.   
This condition fails in case of  the  Kodaria-Thurston surface, which is a $4$-dimensional   nilmanifold, i.e.  a compact quotient of  the nilpotent Lie group ${\rm Nil}^3\times \R$
by a lattice,  where  ${\rm Nil}^3$ denotes the $3$-dimensional  real  Heisenberg group.

 The Kodaira-Thurston  surface is a typical example of a compact almost-K\"ahler $4$-dimensional manifold which does not  admit any  K\"ahler structure.  More precisely, 
  it  is    the total space  of a  principal $T^2$-bundle over a torus $\mathbb{T}^2$ (in our notation $T^2$ denotes the torus on the fibres, while $\mathbb{T}^2$ is the torus at the basis) and  it has an   invariant almost-K\"ahler structure  whose symplectic form vanishes along the fibres of the $T^2$-fibration, where by invariant structure  we mean  a structure  induced by a left-invariant one on  ${\rm Nil}^3\times \R$.  The almost-K\"ahler structures on a total space of a  fibration whose symplectic form vanishes along the fibres are usually called \emph{Lagrangian}, since the fibers are Lagrangian submanifolds. 

In \cite{TW} Tosatti and Weinkove studied the Calabi-Yau equation on the Kodaira-Thurston surface endowed with an invariant   Lagrangian almost-K\"ahler structure, showing the existence of a solution for every $T^2$-invariant normalized volume form $\sigma$.  In \cite{FLSV}  the previous result  obtained by Tosatti and Weinkove  was simplified  and   extended  to other  $T^2$-bundles over a $\mathbb{T}^2$ endowed with an invariant  Lagrangian almost-K\"ahler structure. 

We recall that in  view of \cite{Ue}  every orientable $T^2$-bundle over a $\mathbb{T}^2$  is  a infra-solvmanifold, i.e. a  smooth quotient $\Gamma\backslash G$ covered by
 a solvmanifold $\tilde \Gamma \backslash G$, compact quotient  by a co-compact discrete subgroup of  one of the following four   Lie groups 
$$
\R^4,\quad {\rm Nil}^3\times \R\,,\quad {\rm Nil}^4\,,\quad {\rm Sol}^3\times \R\,.
$$
  These  Lie groups are all diffeomorphic to $\R^4$. The Lie  
groups  ${\rm Nil}^3$, ${\rm Nil}^4$ are nilpotent  and  $Sol^3$ is a particular solvable (non nilpotent)  Lie group.

In particular,  if   the total space $M^4$ of an   orientable  $T^2$-bundle over a $\bT^2$   is a solvmanifold, then  it must be the  compact quotient of one of the above Lie groups $G$.  It is well 
known  that all   the orientable  $T^2$-bundles over  $\bT^2$ admit symplectic structures  (see \cite{Geiges}). The notion of invariant almost-K\" ahler structure makes sense for orientable  $T^2$-bundles over  $\bT^2$, meaning 
one induced from a left-invariant structure on $G$ which is invariant by  the discrete subgroup $\Gamma$.

\smallskip
 As a main result of \cite{FLSV}  it was shown that 
if  $M^4=\Gamma \backslash G$ is an orientable   $T^2$-bundle over a $\mathbb{T}^2$ with $G=\Nil^3\times\R$ or $\Nil^4$, and
if  $M^4$ admits an invariant  Lagrangian almost-K\"ahler structure $(\Omega,J)$, then for every normalized volume form $\sigma=\eE^F \Omega^2$
with $F\in C^\infty(\bT^2)$, the corresponding Calabi--Yau problem has
a unique solution.

The Lagrangian condition may or may not apply in the  case of $G = \Nil^3\times\R$, but is automatic 
when $M^4$ is modelled on  the $3$-step nilpotent Lie group $Nil^4$. In the case of $G = \Sol^3$ every  invariant almost-K\" ahler on  $\Gamma \backslash G$   is non-Lagrangian.

The aim of this paper is to  extend the main result  in \cite{FLSV} to the {\em non-Lagrangian} cases,  i.e.  to   some $T^2$-fibrations modelled on ${\rm Nil}^3 \times \R$ and to   all the 
$T^2$-fibrations modelled on ${\rm Sol}^3\times \R$.  

Our main  result is the following
\begin{theorem}\label{main}
Let  $M^4=\Gamma \backslash G$ be  an orientable   $T^2$-bundle over a $\mathbb{T}^2$ with  $G=\Nil^3\times\R$ or   $\Sol^3\times \R$, and
suppose that $M^4$ admits an invariant  non-Lagrangian almost-K\"ahler structure $(\Omega,J)$.  Then for every normalized volume form $\sigma=\eE^F\Omega^2$
with $F\in C^\infty(\bT^2)$, the corresponding Calabi--Yau problem has
a unique solution.
\end{theorem}
The proof of this theorem consists in  showing that the Calabi-Yau problem  can be reduced to a single elliptic Monge-Amp\`ere equation which has solution. 

\smallskip
The trick of reducing the problem to a Monge-Amp\`ere equation
was the core of \cite{FLSV}, but the class of equations which appear in the present paper differs from the ones considered  in \cite{FLSV}.

As a consequence we show  that  for  every total space $M^4$ of an orientable   $T^2$-bundle over a $\mathbb{T}^2$ endowed with an   invariant  almost-K\"ahler structure $(\Omega,J)$   the Calabi-Yau problem has a solution for every normalized $T^2$-invariant  volume form.
\smallskip

The paper is organized as follows: In Section \ref{pre} we recall the classification of $T^2$-bundles over $\mathbb{T}^2$ and we briefly  describe
the  main result  in  \cite{FLSV}. Sections \ref{sec3} and \ref{sec4}  contain the proof of Theorem \ref{main} where the case of $G=\Nil^3\times\R$ and 
$G=\Sol^3\times \R$ are treated separately.  In each of the two cases we can reduce the problem to a Mong\`e-Ampere equation for which we show the existence of a solution.  

\section{The Calabi-Yau equation on $T^2$-bundles over $\mathbb{T}^2$}\label{pre}
Orientable $T^2$-bundles  over a  $\mathbb{T}^2$ were classified by
Fukuhara and Sakamoto in \cite{Sakamoto-Fukuhara} and it was shown by Ue in \cite{Ue,Ue2} that all these manifolds are
infra-solvmanifolds. A compact manifold $M$ is called an \emph{infra-solvmanifold} if it admits a finite cover $\pi\colon \tilde{M}\to M$, where
$\tilde{M}=\tilde \Gamma \backslash G$
is the compact quotient of a solvable Lie group $G$ by a lattice $\tilde \Gamma$. Alternatively, $M$ can be written as a quotient
$M=\Gamma\backslash G$, where $\Gamma$ is a discrete group containing  a lattice $\tilde \Gamma$ of $G$ such that
$\tilde \Gamma\backslash\Gamma$ is finite. In the case that $\tilde \Gamma$ is a lattice, $M$ is simply called a {\em solvmanifold}.

\medskip
It turns out that in the classification of $T^2$-bundles over $\mathbb{T}^2$, the   solvable Lie group $G$  must be one of  the following four patterns
\begin{equation}\label{list}
\R^4\,,\quad {\rm Nil}^3\times \R\,,\quad {\rm Nil}^4\,,\quad {\rm Sol}^3\times \R\,,
\end{equation}
while the classification of the possible $ \Gamma$'s determines eight families.   For the  Lie groups     ${\rm Nil}^3\times \R\,,  {\rm Nil}^4\,,  {\rm Sol}^3\times \R$ we have the following description:

\begin{itemize}

\vspace{2mm}\item[(1)] $\Nil^3$ is the
$3$-dimensional Heisenberg group of matrices
\begin{equation*}\label{xyz} 
\left(\!\begin{array}{ccc}1&x&z\\0&1&y\\0&0&1\end{array}\!\right)
\end{equation*}
and $\Nil^3\times \R$ is a $2$-step nilpotent Lie group which can be regarded as $\R^4$ with the product 
$$
(x_0, y_0, z_0, t_0) (x, y, z, t) = (x_0 +x ,\ y_0 +y,\ z_0+z+x_0y ,\ t_0 + t)\,.
$$

\vspace{2mm}\item[(2)] $\Nil^4 = \R\ltimes\R^3$ is the $3$-step nilpotent Lie group of
real matrices
\begin{equation*}
\left(\begin{array}{cccc} 1&t&\frac 12
  t^2&x\\ 0&1&t&y\\ 0&0&1&z\\ 0&0&0&1 \end{array} \right ).
\end{equation*}

\vspace{3mm}\item[(3)] $\Sol^3 =\R\ltimes_\varphi\R^2$ is a unimodular $2$-step solvable Lie group with $\varphi(t) =
\hbox{\small$\left(\!\begin{array}{cc} \eE^t &0\\0&\eE^{-t}\end{array}
  \!\!\right)$}$ and $\Sol^3 \times \R$ can be regarded as $\R^4$ with the product 
\[(x_0, y_0, z_0, t_0) (x, y, z, t) = (x_0 + \eE^{t_0} x ,\ y_0 +
  \eE^{-t_0} y,\ z_0 + z,\ t_0 + t).\]
\end{itemize}

The diffeomorphism classes of (the total space of) $T^2$-bundles over
$\bT^2$ can be summarized in Geiges' eight  families
\cite[Table~1]{Geiges}, which can be explicitly described in terms of the generators of the  discrete groups $\Gamma$,  the
monodromy matrices  along the two curves
 generating $\pi_{1}(\bT^2)$, as well as the Euler
class  for the corresponding $T^2$-bundle.

In the case of $G = \Nil^3 \times \R$ one has two inequivalent fibrations
$$
\begin{array}{rcl}
\pi_{xy} &\colon& M^4 \to \bT^2_{xy},\\[5pt]
\pi_{yt} &\colon& M^4 \to \bT^2_{yt}
\end{array}
$$
induced from the
following coordinate mappings:
$$
\begin{array}{l}
(x,y,z,t)\ \mapsto (x,y),\\[5pt]
(x,y,z,t)\ \mapsto  (y,t).
\end{array}
$$

If $\Gamma$ is not a lattice of $G$, we have that 
$\Gamma$   contains a lattice $\tilde \Gamma$ of $G$ such that
the quotient $\tilde \Gamma \backslash \Gamma$ is a finite group. Therefore there exists a
covering map $p\colon  \tilde \Gamma \backslash  G \to \Gamma\backslash G$
which preserves the $T^2$-bundle structure over $\mathbb{T}^2$.

We recall that an {\em almost-K\"ahler structure} on a manifold $M$ is a pair $(\Omega,J)$, where $\Omega$ is a symplectic 
form and $J$ is a endomorphism of the tangent bundle to $M$ satisfying $J^{2}=-I$ and 
$$
\Omega(JX,JY)=\Omega(X,Y)\,,\quad \Omega(Z,JZ)>0
$$ 
for every tangent vector fields $X,Y,Z$ on $M$ with $Z$ nowhere vanishing. Every almost-K\"ahler structure induces 
the Riemannian metric 
$$
g(X,Y)=\Omega(X,JY)\,.
$$
\smallskip 
In this paper (as in \cite{FLSV}) we consider  on the total space  $M^4 = \Gamma \backslash G$ of    $T^2$-bundles over
$\bT^2$    invariant almost-K\"ahler structures, i.e. ones induced from  left-invariant structures on $G$ which are invariant by  the discrete group $\Gamma$ and we study  for these almost-K\"ahler manifolds the Calabi-Yau problem. In particular every
invariant  almost-K\"ahler  structure on $M^4 = \Gamma\backslash G$ induces an
invariant  almost-K\"ahler  structure  on  the solvmanifold $ \tilde \Gamma \backslash G$.

The case $G=\R^4$ (which corresponds to two of the Geiges' families)   is not  interesting from our point of view, since in this case every invariant almost-K\"ahler structure is K\"ahler and Yau's theorem can be applied. For the other cases we have to distinguish the 
Lagrangian case from the non-Lagrangian one. 
If $M^4$ is modelled on $G=\Nil^4$  or  on   $G=\Nil^3\times \R$ with bundle  structure given by the projection $\pi_{xy}$, then every invariant almost-K\"ahler is Lagrangian. If  $M^4$ is modelled on $\Nil^3\times \R$ with bundle structure given by the projection $\pi_{yt}$ then it admits Lagrangian and non-Lagrangian  almost-K\"ahler structures as well. In the case $G=\Sol^3\times \R$, every invariant almost-K\"ahler structure is non-Lagrangian.

 Let now $M^4 = \Gamma/G$ be  an  orientable $T^2$-bundle  over $\mathbb{T}^2$ and denote by $\g$ the Lie algebra of $G$. 
Then every basis $(e^i)$ of the dual space $\g^*$ induces a global frame of $1$-forms on $M^4$.  
Furthermore we fix  an invariant almost-K\"ahler structure 
$(\Omega,J)$ on $M^4$. Let $\sigma={\rm e}^F\,\Omega^2$ be a volume  form  and let $F$ be  a smooth map on the base
$\mathbb{T}^2$ of $M^4$ satisfying 
$$
\int_{\bT^2} (\eE^F-1)=0.
$$  
Then in this case the Calabi-Yau problem reads as 
\begin{equation}\label{torus}
\begin{cases}
(\Omega+da)^2={\rm e}^F\,\Omega^2\,,\\
J(da)=da\,,
\end{cases}
\end{equation} on $M^4$ whose components with respect to the coframe $(e^i)$ are defined on the torus $\bT^2$.  Thus the Calabi-Yau problem reduces to  a system of partial differential equations on the base $\bT^2$.

Although  the system \eqref{torus} depends on the choice of $G$, $(\Omega,J)$ and the structure of $T^2$-fibration, for all the cases we can proceed in the following way:\\
first we  parametrize $(\Omega,J)$ using a suitable invariant coframe on $M^4$  in order to simplify the formulation of  \eqref{torus} as far as possible and then we perform a suitable change of variables transforming the system \eqref{torus} in a Monge-Amp\`ere  equation on the base $\bT^2$.  
 
The Lagrangian cases have been considered in \cite{FLSV}, where  it has been proved that  
has a unique solution. In the next  two sections we will consider the non-Lagrangian cases for the manifolds modelled on $\Nil^3 \times \R$ and $\Sol^3 \times \R$.

\section{Manifolds modelled on ${\rm Nil}^3\times \R$: the non-Lagrangian case} \label{sec3}
In this section we study the Calabi-Yau problem for $T^2$-bundles over  ${\mathbb T}^2$  modelled on ${\rm Nil}^3\times \R$ and equipped with an invariant non-Lagrangian almost-K\"ahler structure.
 
\smallskip  
The structure of $T^2$-bundle over a $\bT^2$    is then  induced by the projection 
$\pi_{yt}$ onto the torus $\bT^2_{yt}$.  The total space $M^4$ of the $T^2$-fibration  has the global invariant coframe 
$$
e^1= dy\,,\quad e^2=dx\,,\quad e^3=dt\,,\quad e^4=dz-xdy
$$
which satisfies the structure equations 
\begin{equation}
\label{struc-eq} de^1=de^2=de^3=0\,,\qquad de^4=e^{12}\,. 
\end{equation}
 
\begin{lemma}\label{inv2}
 Let $M^4$ be the total space of an oriented  $T^2$-bundle over  ${\mathbb T}^2$  modelled on ${\rm Nil}^3\times \R$ and induced by  the projection 
$ \pi_{yt}$ onto the torus $\bT^2_{yt}$.
 
 Let $(\Omega,J)$ be an invariant almost-K\"ahler structure on $M^4$ with induced Riemannian metric $g$.
Then there exists an orthonormal  invariant coframe $(f^i)$ for which
\begin{equation}\label{Omega5}
\Omega=f^{14}+f^{23},
\end{equation}
and 
\begin{equation}
\label{eqn:Y1}
f^1\in \langle e^1\rangle\,,\quad g(e^3,f^2)=0\,,\quad g(e^3,f^3)g(e^3,f^4)\ge 0.
\end{equation}
\end{lemma}
\begin{proof}
We can certainly find an orthonormal invariant coframe $(f^i)$  for which \eqref{Omega5} is valid and $f^1\in \langle e^1\rangle$.
Since $f^4=J(f^1)$, we still have the freedom of rotate $f^{23}$ in the plane orthogonal to $\langle f^1,f^4\rangle$. After a suitable rotation we obtain 
$g(e^3,f^2)=0$. Eventually, we may invert the direction of $f^2$ and $f^3$ to meet the condition  $g(e^3,f^3)g(e^3,f^4)\ge 0$, without reversing $f^{23}$.
\end{proof}
\subsection{The Calabi-Yau equation on $M^4$.} Consider on $M^4$ an invariant non-Lagrangian  almost-K\"ahler structure $(\Omega,J)$ with induced  Riemannian metric $g$. Let $\sigma={\rm e}^F\,\Omega^2$ be a volume form where $F=F(y,t)$ is a smooth map on the 
base satisfying 
\begin{equation}
\label{eqn:Z3}
\int_{\mathbb{T}^2} ({\rm e}^F-1)=0\,.
\end{equation}
Consider the Calabi-Yau  equation 
$$
\begin{cases}
(\Omega+da)^2=\sigma\,,\\
J(da)=da,
\end{cases}
$$
where $a$ is a $1$-form on $M^4$ whose components with respect to the basis $(e^i)$ depend on $(y,t)$ only. 
Let  $(f^i)$ be a coframe as in Lemma \ref{inv2} and set
$$
G^i_j=g(e^i,f^j);
$$  
then we have
$$
e^i=G^i_j f^j.
$$
In particular we have
\begin{equation}
\label{eqn:Z1}
e^1=G^1_1 f^1
\end{equation}
and
\begin{equation}
\label{eqn:Z2}
G^1_1\ne 0,\qquad G^1_2=G^1_3=G^1_4=0.
\end{equation}
Let $H=G^{-1}$ be the inverse matrix of $G = (G^i_j)$. Then 
$$
f^i=H^i_j e^j.
$$
From \eqref{eqn:Z1} and \eqref{eqn:Z2}, we have
\begin{equation*}
H^1_1=(G^1_1)^{-1}\ne 0,
\end{equation*}
and
\begin{equation}
\label{eqn:AA}
H^1_2=H^1_3=H^1_4=0.
\end{equation}
Thanks to structure equations \eqref{struc-eq},  we have
$$
d f^i=H^i_j d e^j=H^i_4 d e^4=H^i_4 e^{12}=
H^i_4 G^1_1 (G^2_2 f^{12}+G^2_3 f^{13}+G^2_4 f^{14}).
$$
The condition that $(\Omega,J)$ is non-Lagrangian implies that
\begin{equation*}
G^3_4\ne 0.
\end{equation*}
Moreover, since $G^3_2=g(e^3,f^2)=0$,  we have 
\begin{equation}\label{2eqn:1}
e^3=G^3_1 f^1+G^3_3 f^3+G^3_4 f^4,
\end{equation}
where 
\begin{equation}
G^3_3G^3_4\ge 0,
\label{eqn:Y2}
\end{equation}
thanks to \eqref{eqn:Y1}.

Differentiating \eqref{2eqn:1} we get  
$$
G^3_3 df^3+G^3_4 df^4=(G^3_3 H^3_4+G^3_4 H^4_4)e^{12}=0\,,
$$
i.e.
\begin{equation}
G^3_3 H^3_4+G^3_4 H^4_4=0.
\label{2eqn:2}
\end{equation}
Furthermore, the symplectic condition $d\Omega=0$ gives 
$$
df^{23}=0
$$
that is
\begin{equation*}
\begin{split}0&=H^2_4G^1_1 (G^2_2 f^{12}+G^2_3 f^{13}+ G^2_4 f^{14})\wedge f^3- 
H^3_4 G^1_1 f^2\wedge (G^2_2 f^{12}+G^2_3 f^{13}+ G^2_4 f^{14})
\\&=
G^1_1(H^2_4G^2_2+G^1_1H^3_4G^2_3)\,f^{123}-G^1_1H^2_4 G^2_4\, f^{134}+G^1_1H^3_4 G^2_4\, f^{124}\,.
\end{split}
\end{equation*}
It follows that
\begin{equation}
H^2_4G^2_2+H^3_4G^2_3=0,\qquad H^2_4G^2_4=H^3_4G^2_4=0.
\label{2eqn:3}
\end{equation}
From \eqref{2eqn:2} and \eqref{2eqn:3} we have
\begin{equation}
\label{eqn:BB}G^2_4G^3_4H^4_4=G^2_4(G^3_3H^3_4+G^3_4H^4_4)=0,
\end{equation}
hence, since $H^1_4=0$, and $H^2_4$, $H^3_4$ and $H^4_4$ cannot vanish all together, from \eqref{eqn:AA}, \eqref{2eqn:3} and \eqref{eqn:BB} we obtain that
\begin{equation*}
G^2_4=0.
\end{equation*}
Write  
$$
a=a_k f^k
$$
and compute  
$$
\begin{aligned}
da
=& (G^1_1a_{2,y}+G^1_1G^2_2H^k_4a_k +G^3_1a_{2,t}) f^{12}+(G^1_1a_{3,y}+ G^1_1G^2_3H^k_4a_k+ G^3_1a_{3,t}- G^3_3a_{1,t}) f^{13}\\
&- G^3_4a_{2,t} f^{24}+(G^1_1a_{4,y}+ G^3_1a_{4,t}-G^3_4a_{1,t}) f^{14}- G^3_3a_{2,t}  f^{23}+( G^3_3a_{4,t}-G^3_4a_{3,t}) f^{34}.
\end{aligned}
$$
Hence $da$ is of type $(1,1)$ with respect to $J$ if and only if 
$$
\begin{cases}
G^1_1a_{2,y}+G^1_1G^2_2(H^2_4 a_2+H^3_4 a_3+H^4_4 a_4) +G^3_1a_{2,t}=-G^3_3a_{4,t} +G^3_4a_{3,t}, \\
G^1_1a_{3,y}+G^1_1G^2_3(H^2_4 a_2+H^3_4 a_3+H^4_4 a_4) +G^3_1 a_{3,t} -G^3_3a_{1,t} =- G^3_3a_{2,t} 
\end{cases}
$$
and in this case $da$ reduces to 
$$
\begin{aligned}
da=& (- G^3_3 a_{4,t}+G^3_4a_{3,t}) f^{12}- G^3_3 a_{2,t} f^{13}- G^3_4a_{2,t} f^{24}\\
&+(G^1_1a_{4,y}+ G^3_1 a_{4,t}-G^3_4a_{1,t}) f^{14}- G^3_3 a_{2,t} f^{23}+(G^3_3 a_{4,t} -G^3_4a_{3,t}) f^{34}.
\end{aligned}
$$ 
The Calabi-Yau equation now reads as  
\begin{equation*}
\mathrm e^F=(1+G^1_1a_{4,y}+ G^3_1 a_{4,t}-G^3_4a_{1,t})(1- G^3_3 a_{2,t})- G^3_3G^3_4 (a_{2,t})^2-(- G^3_3 a_{4,t} +G^3_4a_{3,t})^2
\end{equation*}
and the Calabi-Yau problem is equivalent to the  following system of partial differential equations: 
\begin{equation}
\label{2eqn:4}\begin{cases}
G^1_1a_{2,y}+G^1_1G^2_2(H^2_4 a_2+H^3_4 a_3+H^4_4 a_4) +G^3_1a_{2,t}+G^3_3a_{4,t} -G^3_4a_{3,t}=0,\\
G^1_1a_{3,y}+G^1_1G^2_3(H^2_4 a_2+H^3_4 a_3+H^4_4 a_4) +G^3_1 a_{3,t} -G^3_3a_{1,t} + G^3_3a_{2,t}=0, \\
\begin{aligned}(1+G^1_1a_{4,y}+ G^3_1 a_{4,t}-G^3_4a_{1,t})&(1- G^3_3 a_{2,t})-
\\&- G^3_3G^3_4 (a_{2,t})^2-(- G^3_3 a_{4,t} +G^3_4a_{3,t})^2=\mathrm e^F.
\end{aligned}
\end{cases}
\end{equation}
 In the system \eqref{2eqn:4} the parameter $G^3_3$ has a special role.  We will study separately the cases $G^3_3 =0$ and $G^3_3\ne 0$.

\subsection{The case $G^3_3$=0.}   This  case is quite trivial since condition $G^3_3=0$ implies
$dt\in \langle f^1,f^4\rangle$ and $f^{14}=dy\wedge dt$.  Therefore if $G^3_3=0$ the corresponding Calabi-Yau equation has the explicit solution 
$$
\tilde{\Omega}=({\rm e}^F-1) f^{14}+f^{23}\,.
$$

\subsection{The Case $G^3_3\ne 0$.}
Under this assumption 
 we consider the transformation 
\begin{align*}
& a_1=-G^3_3u_t-G^1_1(H^2_4G^2_3+H^4_4G^2_2)u, \\
& a_2=-G^3_3 u_t- H^4_4 G^1_1G^2_2 u, \\
& a_3= -H^4_4 G^1_1G^2_3 u, \\
&a_4 = G^1_1u_y+ G^3_1 u_t.
\end{align*}
A long but straightforward computation shows that  the first two equations of system \eqref{2eqn:4} are identically satisfied, while the third one  becomes
\begin{equation}\label{2eqn:6}
(u_{yy}+B_{11} u_t + C_{11})(u_{tt}+B_{22} u_t + C_{22})- (u_{yt}+B_{12} u_t + C_{12})^2=E_1+E_2\mathrm e^F,
\end{equation}
where
\begin{align*}
B_{11}&=\frac {G^2_2(G^3_1)^2 H^4_4}{G^1_1G^3_3}-2\frac {G^2_3G^3_1G^3_4H^4_4}{G^1_1G^3_3}+\frac{G^2_3G^3_4H^2_4}{G^1_1}, \\
B_{12}&=-\frac{G^2_2G^3_1H^4_4}{G^3_3}+\frac {G^2_3G^3_4H^4_4}{G^3_3}, \qquad
B_{22}=\frac{G^1_1G^2_2H^4_4}{G^3_3}, \\
C_{11}&=\frac {1}{(G^1_1)^2}+\frac {(G^3_1)^2}{(G^1_1)^2(G^3_3)^2}+\frac {G^3_4}{(G^1_1)^2G^3_3}, \\
C_{12}&=-\frac {G^3_1}{G^1_1(G^3_3)^2},\qquad
C_{22}=\frac {1}{(G^3_3)^2}, \\
E_1&=\frac {G^3_3G^3_4}{(G^1_1)^2(G^3_3)^4},\qquad E_2=(G^1_1 G^3_3)^2.
\end{align*}

In particular we have
\begin{equation*}
B_{11}B_{22}-(B_{12})^2=0
\end{equation*}
and
\begin{equation*}
C_{11}C_{22}-(C_{12})^2=E_1+E_2.
\end{equation*}

\section{Manifolds modelled on  ${\rm Sol^3}\times \R$}\label{sec4}
In this section we study the Calabi-Yau equation for  the total spaces $M^4$ of $T^2$-bundles over torus ${\mathbb T}^2$  modelled on  ${\rm Sol^3}\times \R$. 

\medskip

Since the Lie group ${\rm Sol}^3\times\R$ can be seen as $\R^4$ equipped with the product
$$
(x_0, y_0, z_0, t_0) (x, y, z, t) = (x_0 +x ,\ y_0 +y,\ z_0 + \eE^x z,\ t_0 + \eE^{-x} t)
$$
$M^4$ inherits the global  invariant coframe
\begin{equation}\label{1341}\textstyle
e^1 = dx,\qquad e^2 = dy,\qquad e^3 = \eE^x\,dz,\qquad e^4 =
\eE^{-x}dt
\end{equation}
satisfying the following structure equations
\begin{equation}
\label{eqn:X3}
d e^1=de^2=0,\qquad d e^3=e^{13},\qquad  d e^4= -e^{14}.
\end{equation}

Moreover, invariant almost-K\"ahler structure on $M^4$ should be parametrized as claimed in the following lemma proved in \cite{FLSV}

\begin{lemma}\label{inv}
Let $(\Omega,J)$ be an invariant almost-K\"ahler structure on  the total space $M^4$ of a  $T^2$-bundle over a torus ${\mathbb T}^2$  modelled on  ${\rm Sol^3}\times \R$. Let $g$ be the induced Riemannian metric.
Then there exists an orthonormal  global coframe $(f^i)$ for which
\begin{equation*}
\Omega=f^{12}+f^{34},
\end{equation*}
and
\begin{equation}
\label{eqn:W1}
f^1\in \langle e^1\rangle\,,\quad f^3\in\langle e^3\rangle\,,\quad f^{4}\in\langle e^3,e^4\rangle,
\end{equation}
with
\begin{equation}
\label{eqn:W2}
g(e^1,f^1)>0.
\end{equation}
\end{lemma}
Notice that in this case every invariant almost-K\"ahler structure is non-Lagrangian. 
\subsection{The Calabi-Yau equation on $M^4$.}
Let $(\Omega,J)$ be an invariant almost-K\"ahler structure on $M^4$, $(f^k)$ be a coframe as in the previous lemma and $\sigma=\eE^F\,\Omega^2$
be a volume form where $F=F(x,y) \in C^{\infty} ({\mathbb{T}^2})$  satisfies the condition
$$
\int_{\mathbb{T}^2} ({\rm e}^F-1)=0\,.
$$
Then we consider the Calabi-Yau equation 
\begin{equation}\label{CYsol}
(\Omega+da)^2=\sigma,
\end{equation}
where 
$$
a =\sum_{k=1}^4 a_k f^k\,,
$$
is a $1$-form whose components $a_k$ are functions on the base ${\mathbb T}^2$.

Let $g$ be the Riemannian metric induced by $(\Omega,J)$ and set
$$
G^i_j=g(e^i,f^j).
$$  
 Then 
\begin{equation}
\label{eqn:X4}
e^i=G^i_jf^j
\end{equation}
and
$$
f^i=H^i_j e^j,
$$
where  $H=G^{-1}$ is the inverse matrix of $G = (G^i_j)$.
In particular this implies that
\begin{equation*}
G^3_4=H^2_4G^2_2+H^4_4G^2_4=H^3_3G^4_3+H^4_3G^4_4=0.
\end{equation*}

From \eqref{eqn:W1} we have that
$$
f^1= H^1_1\,e^1 \,,\quad f^2=H^2_2\, e^2+H^2_3\,e^3+H^2_4 e^4\,,
\quad f^3=H^3_3\, e^3\,,\quad f^4=H^4_3\, e^3+H^4_4\,e^4\,.
$$
Then, making use of \eqref{eqn:X3} and \eqref{eqn:X4}, we obtain     
\begin{alignat*}{2}
&df^1=0,&\quad &df^2=G^1_1(H^2_3G^3_3-H^2_4G^4_3)f^{13}-G^1_1H^2_4G^4_4f^{14},
\\ &df^3=G^1_1f^{13},&\quad& df^4=G^1_1(H^4_3G^3_3-H^4_4G^4_3)f^{13}-G^1_1f^{14}.
\end{alignat*}

Let 
$$
a=a_1f^1+a_2f^2+a_3f^3+a_4f^4
$$
 be a $T^2$-invariant form on $M^4$.

We have
\begin{align*}
d a=& (G^1_1a_{2,x}-G^2_2 a_{1,y})f^{12}
\\[4 pt] &
+\Bigl(G^1_1a_{3,x}-G^2_3a_{1,y}+G^1_1(H^2_3G^3_3-H^2_4G^4_3)a_2+G^1_1a_3+G^1_1(H^4_3G^3_3-H^4_4G^4_3)a_4\Bigr)f^{13}
\\[4 pt] &
+\Bigl(G^1_1 a_{4,x}-G^2_4a_{1,y}-G^1_1H^2_4G^4_4a_2-G^1_1a_4\Bigr)f^{14}
\\[4 pt] &
+(G^2_2a_{3,y}-G^2_3a_{2,y})f^{23}+(G^2_2a_{4,y}-G^2_4a_{2,y})f^{24}+(G^2_3a_{4,y}-G^2_4a_{3,y})f^{34}\,.
\end{align*}
Hence $d a$ is $J$-invariant if and only if its components satisfy
\begin{equation*}
\begin{cases}
\begin{aligned}
G^1_1a_{3,x}-G^2_3a_{1,y}+G^1_1(H^2_3G^3_3-H^2_4G^4_3)a_2+G^1_1a_3+&G^1_1(H^4_3G^3_3-H^4_4G^4_3)a_4
\\ & =
G^2_2a_{4,y}-G^2_4a_{2,y},
\end{aligned}
\\[4 pt]
G^1_1 a_{4,x}-G^2_4a_{1,y}-G^1_1H^2_4G^4_4a_2-G^1_1a_4=G^2_3a_{2,y}-G^2_2a_{3,y},
\end{cases}
\end{equation*}
and equation \eqref{CYsol} becomes
\begin{multline*}
(1+G^1_1a_{2,x}-G^2_2a_{1,y})(1+G^2_3a_{4,y}-G^2_4a_{3,y})
\\
-(G^2_2a_{4,y}-G^2_4a_{2,y})^2-(G^2_2a_{3,y}-G^2_3a_{2,y})^2
={\rm e}^F\,.
\end{multline*}
Therefore the  Calabi-Yau problem reduces to the following system of partial differential equations
\begin{equation}\label{system}
\begin{cases}
\begin{aligned}
G^1_1a_{3,x}-G^2_3a_{1,y}+G^1_1(H^2_3G^3_3-H^2_4G^4_3)a_2+G^1_1a_3+&G^1_1(H^4_3G^3_3-H^4_4G^4_3)a_4
\\ & =
G^2_2a_{4,y}-G^2_4a_{2,y},
\end{aligned}
\\[4 pt]
G^1_1 a_{4,x}-G^2_4a_{1,y}-G^1_1H^2_4G^4_4a_2-G^1_1a_4=G^2_3a_{2,y}-G^2_2a_{3,y},
\\[4 pt]
\begin{aligned}
(1+G^1_1a_{2,x}-G^2_2a_{1,y})&(1+G^2_3a_{4,y}-G^2_4a_{3,y})
\\
&-(G^2_2a_{4,y}-G^2_4a_{2,y})^2-(G^2_2a_{3,y}-G^2_3a_{2,y})^2
={\rm e}^F\,.
\end{aligned}
\end{cases}
\end{equation}

\subsection{Reduction of \eqref{system} to a single equation.} 
Consider $u\in C^2(\mathbb T^2)$ such that
\begin{equation}
\label{eqn:X2}
\int_{\mathbb T^2} u=0,
\end{equation}
and let
\begin{equation*}
\begin{cases}
\begin{aligned}
&a_1(x,y)=
-\frac {H^1_1G^2_2\bigl(u_y(x,y)-u_y(x,0)\bigr)+
2H^4_4G^4_3\bigl(u(x,y)-u(x,0)\bigr)}{(G^2_3)^2+(G^2_4)^2}
\\&\quad
-\frac {G^1_1H^2_2}{(G^2_3)^2+(G^2_4)^2}\biggl(\int_0^y\bigl(u_{xx}(x,t)-u(x,t)\bigr)\,dt-y\int_0^1\bigl(u_{xx}(x,t)-u(x,t)\bigr)\,dt\biggr),
\end{aligned} 
\\[30pt]
a_2(x,y)={\displaystyle -\frac 1{(G^2_3)^2+(G^2_4)^2}\biggl(\int_0^x\Bigl(\int_0^1 u(s,t)\,dt\Bigr)\,ds-\int_0^1\bigl(u_x(x,t)-u_x(0,t)\bigr)\,dt\biggr),}
\\[15pt]
\begin{aligned}
&a_3(x,y)=
-\frac {H^2_2G^2_3u_x(x,y)+H^1_1G^2_4u_y(x,y)-H^2_2(G^2_3-2G^2_4H^4_4G^4_3)u(x,y)}{(G^2_3)^2+(G^2_4)^2}
\\
&\quad
-\frac{H^2_2G^2_3}{(G^2_3)^2+(G^2_4)^2}\biggl(\int_0^x\Bigl(\int_0^1u(s,t)\,dt\Bigr)\,ds
-\int_0^1\bigl(u_x(x,t)-u_x(0,t)\bigr)\,dt
\biggr),
\end{aligned}
\\[30pt]
\begin{aligned}
&a_4(x,y)=
-\frac {H^2_2G^2_4 u_x(x,y)-H^1_1G^2_3u_y(x,y)+H^2_2G^2_4u(x,y)}{(G^2_3)^2+(G^2_4)^2}
\\&\quad
-\frac {H^2_2G^2_4}{(G^2_3)^2+(G^2_4)^2}\biggl(\int_0^x\Bigl(\int_0^1 u(s,t)\,dt\Bigr)\,ds
- \int_0^1\bigl(u_x(x,t)-u_x(0,t)\bigr)\,dt
\biggr).
\end{aligned}
\end{cases}
\end{equation*}
Thanks to condition \eqref{eqn:X2} we have that the functions $a_1$ to $a_4$ are periodic. 
A long computation shows that  the first two equations of system \eqref{system} are identically satisfied, while the third one becomes:
\begin{equation}
\label{eqn:X5}
(u_{xx}+B_{11}u_y+C_{11}+Du)
(u_{yy}+B_{22}u_y+C_{22})-(u_{xy}+B_{12}u_y)^2=E_1+E_2\mathrm e^F
\end{equation}
where
\begin{align*}
&B_{11}=\frac {2\,H^1_1G^2_2G^2_3(G^2_4+G^2_3H^4_4G^4_3)}{(G^2_3)^2+(G^2_4)^2}, \\
&B_{12}=\frac {(G^2_4)^2-(G^2_3)^2+2G^2_3G^2_4H^4_4G^4_3}{(G^2_3)^2+(G^2_4)^2}, \\
&B_{22}=-\frac {2G^1_1H^2_2G^2_4(G^2_3-G^2_4H^4_4G^4_3)}{(G^2_3)^2+(G^2_4)^2}, \\
&C_{11}=H^1_1\Bigl((G^2_2)^2+(G^2_3)^2+(G^2_4)^2\Bigr),\qquad C_{22}=G^1_1,  \\
&D=-1,\qquad E_1=(G^2_2)^2,\qquad E_2=(G^2_3)^2+(G^2_4)^2.
\end{align*}
In particular we have
\begin{equation*}
B_{11}B_{22}-(B_{12})^2=-1
\end{equation*}
and
\begin{equation*}
C_{11}C_{22}-(C_{12})^2=E_1+E_2.
\end{equation*}

\section{The Monge-Amp\`ere equation}

Both equations \eqref{2eqn:6} and \eqref{eqn:X5} are generalized Monge-Amp\`ere equations of the following type:
\begin{equation}
\label{eqn:X15}
A_{11}[u]A_{22}[u]-\bigl(A_{12}[u]\bigr)^2=E_1+E_2\, \mathrm e^F,
\end{equation}
where
\begin{align*}
&A_{11}[u]=u_{xx}+B_{11}u_y+C_{11}+Du, \\
&A_{12}[u]=u_{xy}+B_{12}u_y+C_{12}, \\
&A_{22}[u]=u_{yy}+B_{22}u_y+C_{22},
\end{align*}
with $B_{ij}$, $C_{ij}$, $D$, $E_i$ real numbers such that
\begin{gather}
\label{eqn:X16}
C_{11}+C_{22}>0,
\qquad D\le 0,
\\
\label{eqn:X24} E_1>0,\qquad E_2>0,
\\
\label{eqn:XX24}
B_{11}B_{22}-(B_{12})^2=D
\end{gather}
and 
\begin{equation}
\label{eqn:X17}
C_{11}C_{22}-(C_{12})^2=E_1+E_2.
\end{equation}
Moreover 
\begin{equation}
\label{eqn:X25}
F\in C^\infty(\mathbb T^2)
\end{equation}
 and satisfies the condition
\begin{equation}
\label{eqn:X26}
\int_{\mathbb T^2}(\mathrm e^F-1)=0.
\end{equation}
In Theorem \ref{thm:X2} we shall prove  that  equation \eqref{eqn:X15} has a solution belonging to $C^\infty(\mathbb T^2)$ and satisfying the condition
\begin{equation}
\label{eqn:X18}
\int_{\mathbb  T^2}u=0.
\end{equation}

For all $n\in\mathbb N$, $0<\epsilon<1$, consider the semi-norms
\begin{gather*}
\abs{u}_{C^n}=
\max_{\scriptscriptstyle 0\le j\le n}
\sup_{\scriptscriptstyle (x,y)\in \mathbb R^2}
\abs{\partial^j_x\partial^{n-j}_y u(x,y)}
\\
\abs{u}_{C^{n,\epsilon}}=\max_{\scriptscriptstyle 0\le j\le n}
\sup_{\scriptscriptstyle (x,y)\in \mathbb R^2}
\sup_{\scriptscriptstyle (h,k)\in\mathbb R^2\setminus\{(0,0)\}}\frac {\bigabs{\partial^j_x\partial^{n-j}_y u(x+h,y+k)-\partial^j_x\partial^{n-j}_y u(x,y)}}{(h^2+k^2)^{\epsilon/2}}
\intertext{and the norms}
\norm{u}_{C^n}=\max_{\scriptscriptstyle 0\le k\le n}\abs{u}_{C^k},\qquad \norm{u}_{C^{n,\epsilon}}=\max\bigl\{\norm{u}_{C^n},\,\abs{u}_{C^{n,\epsilon}}\bigr\}.
\end{gather*}

\begin{lemma}\label{lem:X2}
Under the hypotheses \eqref{eqn:X24}, for all $u\in C^2(\mathbb T^2)$ satisfying \eqref{eqn:X15} we have that
\begin{equation}
\label{eqn:X20}
\begin{cases}
A_{11}[u] >0, \\ A_{22}[u]>0.\end{cases}
\end{equation}
\end{lemma}
\begin{proof}
Equation  \eqref{eqn:X15} implies that $A_{11}[u]A_{22}[u]>E_1$. Then $A_{11}[u]$ and $A_{22}[u]$ never vanish and have the same sign. At a  point where  $u$ reaches its minimum value, we have $u_y=0$ and $u_{yy}\le 0$. Then 
\begin{equation*}
A_{22}[u]=u_{yy}+C_{22}>0
\end{equation*}
and both $A_{11}[u]$ and $A_{22}[u]$ must be positive everywhere.
\end{proof}
\begin{lemma}
Consider a function $u\in C^2(\mathbb T^2)$ satisfying equation \eqref{eqn:X15}. Under the hypotheses \eqref{eqn:X16}  and \eqref{eqn:X24}  we have that
\begin{equation}
\label{eqn:X21}
Du(x,y)\ge C_{11},\qquad\forall (x,y)\in\mathbb R^2.
\end{equation}
\end{lemma}
\begin{proof}
Consider a point where $Du$ attains its minimum value. Since $D\le 0$, this corresponds to a point where $u$ reaches its maximum value.
Then we have $u_y=0$ and $u_{xx}\le 0$ and from \eqref{eqn:X20} we have
\begin{equation*}
C_{11}+Du\ge u_{xx}+C_{11}+Du>0,
\end{equation*}
which implies 
\begin{equation*}
Du\ge C_{11},
\end{equation*}
at the maximum and therefore everywhere.
\end{proof}

We need  Lemma 6.3 of  \cite{FLSV}:
\begin{lemma}\label{lem:X1}
Consider $w\in C^2(\mathbb T)$ and two real numbers $\alpha$ and $\beta$ such that
\begin{equation}
\label{eqn:X6} w''(t)+\alpha w'(t)\ge \beta,\qquad\forall t\in\mathbb R.
\end{equation} 
Then we have
\begin{equation}
\label{eqn:X14}
\abs{w'(t)}\le 2\abs{\beta}\mathrm e^{2\abs{\alpha}},\qquad\forall t\in\mathbb R.
\end{equation}
\end{lemma}

\begin{theorem}
Assume hypotheses \eqref{eqn:X16} and \eqref{eqn:X17} are satisfied. Then all solutions of \eqref{eqn:X15} satisfy the following estimate:
\begin{equation*}
\norm{u}_{C^2}\le
2\bigl(\abs{B_{11}}+1\bigr)\abs{B_{22}}\mathrm e^{2 C_{22}}+C_{11}+C_{22}.
\end{equation*}\end{theorem}
\begin{proof}
From \eqref{eqn:X20}  we obtain that
\begin{equation*}
u_{yy}+B_{22}u_y\ge -C_{22},
\end{equation*}
hence from Lemma \ref{lem:X1} we obtain that
\begin{equation}
\label{eqn:X22}
\abs{u_y}\le 2 \abs{B_{22}}\mathrm e^{2C_{22}}.
\end{equation}
From \eqref{eqn:X20}, \eqref{eqn:X21} and \eqref{eqn:X22}, we obtain
\begin{equation*}
u_{xx}\ge -2\abs{B_{11}B_{22}}\mathrm e^{2C_{22}}-C_{11}-C_{22},
\end{equation*} 
hence from Lemma \ref{lem:X1} we obtain
\begin{equation}
\label{eqn:X23}\abs{u_x}\le 2\abs{B_{11}B_{22}}\mathrm e^{2C_{22}}+C_{11}+C_{22}.
\end{equation}

Now consider a point $(x_0,y_0)\in [0,1]\times[0,1]$ where  $u$ vanishes. Then we have
\begin{equation*}
\begin{split}
u(x,y)=&\int_0^1 u_x\bigl((1-t)x+tx_0,(1-t)y+ty_0\bigr)\,dt\,(x-x_0)  \\
&+\int_0^1 u_y\bigl((1-t)x+tx_0,(1-t)y+ty_0\bigr)\bigr)\,dt\,(y-y_0),
\end{split}\end{equation*}
which, together periodicity, implies 
\begin{equation*}
\abs{u}_{C^0}\le 2\abs{u}_{C^1}.
\end{equation*}
This estimate, together \eqref{eqn:X22} and \eqref{eqn:X23}, implies
\begin{equation*}
\abs{u}_{C^0}\le 2\bigl(\abs{B_{11}}+1\bigr)\abs{B_{22}}\mathrm e^{2 C_{22}}+C_{11}+C_{22}. \qedhere
\end{equation*}
\end{proof}

Let $\tau\in [0,1]$ and set
\begin{equation}
\mathfrak S_\tau=\Bigl\{u\in C^{2}(\mathbb T^2)\mid A_{11}[u]A_{22}[u]-\bigl(A_{12}[u]\bigr)^2=E_1+(1-\tau)E_2+\tau E_2 \mathrm e^F,\, \textstyle{\int_{\mathbb T^2}u=0}\Bigr\}.
\end{equation}

\begin{theorem}\label{thm:X1}
Assume hypotheses \eqref{eqn:X16} to \eqref{eqn:X25} are satisfied.  Then 
\begin{equation*}
\mathfrak S_\tau\subset C^{2,1/2}(\mathbb T^2),\qquad\forall \tau\in[0,1],
\end{equation*}
and
\begin{equation*}
\sup_{0\le\tau\le 1}\sup_{u\in\mathfrak S_\tau}\norm{u}_{C^{2,1/2}}<\infty.
\end{equation*}
\end{theorem}
\begin{proof}
Thanks to lemma \ref{lem:X2} and hypothesis \eqref{eqn:X24} the equation
\begin{equation*}
A_{11}[u]A_{22}[u]-\bigl(A_{12}[u]\bigr)^2=E_1+(1-\tau)E_2+\tau E_2  \mathrm e^F
\end{equation*}
is uniformly elliptic and we can apply Theorem 2 of \cite{Heinz}.
\end{proof}
\begin{cor}\label{cor:X1}
Under the same hypotheses of Theorem \ref{thm:X1} we have that 
\begin{equation*}
\mathfrak S_\tau\subset C^\infty(\mathbb T^2)
\end{equation*}
for all $0\le \tau\le 1$.
\end{cor}
\begin{proof}
It follows from Theorems 1 and 3 of \cite{Nirenberg}.
\end{proof}

\begin{theorem}\label{thm:X2}
Under the hypotheses \eqref{eqn:X16} to \eqref{eqn:X26},  we have that 
for all $\tau\in [0,1]$ the equation
\begin{equation*}
A_{11}[u]A_{22}[u]-\bigl(A_{12}[u]\bigr)^2=E_1+(1-\tau)E_2+\tau E_2 \mathrm e^F
\end{equation*}
has a solution in $C^\infty(\mathbb T^2)$ satisfying condition \eqref{eqn:X18}.

In particular, for $\tau=1$ we obtain that equation \eqref{eqn:X15} is solvable.
\end{theorem}
\begin{proof}
In view of Corollary  \ref{cor:X1}, 
it is sufficient to prove the existence of a $C^2$-solution.
Hence we have to prove that $\mathfrak S_{\tau}(\mathbb T^2)\ne \emptyset$ for all $\tau\in[0,1]$.
If $\tau=0$, thanks to \eqref{eqn:X17} we have that $0\in\mathfrak S_0$. Then we may set 
\begin{equation*}
\rho=\sup\bigl\{\sigma\in[0,1]\mid \mathfrak S_\tau\ne\emptyset,\,\forall \tau\in[0,\sigma]\bigr\}
\end{equation*}
We must show that $\mathfrak S_\rho\ne\emptyset$ and that $\rho=1$.

Consider a sequence $\tau_n\in [0,\rho]$
converging to $\rho$ and such that 
$\mathfrak S_{\tau_n}\ne \emptyset$. Let $u_n\in \mathfrak S_{\tau_n}$ for all $n$. By Theorem \ref{thm:X1}, the sequence $(u_n)$ is bounded in $C^{2,1/2}(\mathbb T^2)$, hence, by Ascoli-Arzel\`a theorem, it contains a  subsequence $(v_n)$ which converges in $C^2(\mathbb T^2)$ to a function $v$, which is a solution belonging to $\mathfrak S_\rho$.

Now we show that $\rho=1$. Assume by contradiction $\rho<1$ and let $ C^{k,1/2}_\ast(\mathbb T^2)$ be the space of functions $u\in  C^{k,1/2}(\mathbb T^2)$  satisfying $\int_{\mathbb T^2}u=0$. Consider the map
\begin{equation*}
T: C^{2,1/2}_\ast(\mathbb T^2)\times[0,1]\to C_\ast^{0,1/2}(\mathbb T^2),
\end{equation*}
defined as
$$
T(u,\tau)=A_{11}[u]A_{22}[u]-\bigl(A_{12}[u]\bigr)^2-E_1-(1-\tau)E_2-\tau E_2\mathrm e^{F}.
$$
Observe that 
\begin{equation*}
\int_{\mathbb T^2}T(u,\tau)=0,
\end{equation*}
thanks to \eqref{eqn:XX24},  \eqref{eqn:X17} and \eqref{eqn:X26}.

We know that there exists $v\in\mathfrak S_\rho\subset C^{2,1/2}_\ast(\mathbb T^2)$ such that  $T(v,\rho)=0$.
We have that
\begin{equation*}
T'[v,\rho](w,0)=
Lw,
\end{equation*}
with
\begin{equation*}
L:C^{2,1/2}_\ast(\mathbb T^2)\to C_\ast^{0,1/2}(\mathbb T^2)
\end{equation*}
given by
\begin{equation}
\label{eqn:X30}\begin{split}
&Lw=\bigl(A_{22}[v]+C_{22}\bigr)w_{xx}-2\bigl(A_{12}[v]+C_{12}\bigr)w_{xy}+
\bigl(A_{11}[v]+C_{11}\bigr)w_{yy} 
\\&\quad+
\Bigl(B_{11}\bigl(A_{22}[v]+C_{22}\bigr)-2B_{12}\bigl(A_{12}[v]+C_{12}\bigr)
+B_{22}\bigl(A_{11}[v]+C_{11}\bigr)\Bigr)w_{y}
\\&\quad+D\bigl(A_{22}[u]+C_{22}\bigr)w.
\end{split}
\end{equation}

Now from Lemma \ref{lem:X2} and hypotheses \eqref{eqn:X24} the matrices 
\begin{equation*}
\begin{bmatrix}
A_{11}[v] & A_{12}[v] \\ A_{12}[v] & A_{22}[v]\end{bmatrix},\qquad 
\begin{bmatrix}
C_{11} & C_{12} \\ C_{12} & C_{22}\end{bmatrix}
\end{equation*}
are positive, so  their sum is positive too, and the operator $L$  is uniformly elliptic. Since $D\bigl(A_{22}[u]+A_{22}[v]\bigr)\le 0$, we may apply the strong maximum principle (\cite{Gilbarg-Trudinger}, Theorem 3.5)
and obtain  that $Lw=0$ implies that $w$ is constant, that is $w=0$, by  condition \eqref{eqn:X18}.
Ellipticity and classical Schauder estimates (\cite{Gilbarg-Trudinger}, Theorem 6.2) show that $L$ is onto. Since $L$ is one-to-one, it must be an isomorphism.
 Then by the implicit
function theorem there exists $\epsilon>0$ such that $\mathfrak S_\tau(\mathbb T^2)\ne\emptyset$ for $\rho<\tau<\rho+\epsilon$,  in contradiction with the definition of $\rho$.
\end{proof}

\bigbreak

\footnotesize\parindent0pt\parskip8pt

Ernesto Buzano\\
Dipartimento di Matematica, Universit\`a di Torino, Via
Carlo Alberto 10, 10123 Torino, Italia\\
\texttt{ernesto.buzano@unito.it}

Anna Fino\\
Dipartimento di Matematica, Universit\`a di Torino, Via
Carlo Alberto 10, 10123 Torino, Italia\\
\texttt{annamaria.fino@unito.it}

Luigi Vezzoni\\
Dipartimento di Matematica, Universit\`a di Torino,
Via Carlo Alberto 10, 10123 Torino, Italia\\
\texttt{luigi.vezzoni@unito.it}

\end{document}